\documentclass[12pt,a4paper]{article}
\usepackage{amssymb,amsmath,amsthm}

\newtheorem*{theorem}{\indent {\sc Theorem}}
\newtheorem{lemma}{\indent {\sc Lemma}}
\begin{document}

\begin{center}
\textbf {\large{Waring's problem in the natural numbers \\ with binary expansions of a special type}}
\end {center}
\bigskip

\begin{center}
\textbf {\large {K.M. \'Eminyan}}
\end{center}

\begin{center}
{Financial  University  under the Government  of the Russian Federation and\\ 
Bauman State Technical University, Moscow.}\\
Email: eminyan@mail.ru
\end{center}

\bigskip

\begin{abstract}
Let $\mathbb{N}_0$ be a class of natural numbers whose binary expansions contain even numbers of ones. Waring problem in numbers of class $\mathbb{N}_0$ is solved.
\end{abstract}

\begin{center}
{\bf{1. Introduction}}
\end{center}
\bigskip

Let $ n = e_ {0} + e_ {1} 2 + \ldots + e_ {k} 2 ^ {k} $ be a binary expansion of a natural $n$, $(e_ {j} = 0, 1)$. Let $ \mathbb{N}_{0} $ be the set of natural numbers whose binary expansions have an even number of ones, $\mathbb{N}_{1} = \mathbb {N} \setminus \mathbb {N}_{0} $. Define symbol $\varepsilon(n)$
$$
\varepsilon (n) = \left\{
                    \begin{array}{ll}
                      {~~}1, & \hbox{if $n\in \mathbb{N}_0$;} \\
                      -1, & \hbox{if $n\in \mathbb{N}_1$.}
                    \end{array}
                  \right.
$$

In 1969,  A.O. Gelfond \cite{Gelfond} proved that natural numbers of classes of $\mathbb{N}_{0}$ ~ and ~ $\mathbb{N}_{1}$ are regularly distributed in arithmetic progressions.

In 1991, author obtained \cite{Eminyan} an asymptotic formula for the sum of
$$
\sum\limits_{n \leqslant x, \, \, n \in \mathbb{N}_{0}} \tau (n)
$$
and so solved the problem of Dirichlet in numbers from $\mathbb {N}_{0}$.

In this paper, we solve the problem of Waring in numbers from $\mathbb {N}_{0}$.  Note that essentially  a few more general problem is solved. Let $(i_ {1}, \ldots, i_ {k}) $ be an arbitrary set of zeros and ones. We obtain an asymptotic formula for the number of solutions of the equation
\begin{equation}\label{Var1}
    x_ {1} ^ {n} + x_ {2} ^ {n} + \ldots + x_ {k} ^ {n} = N
\end{equation}
in positive integers $ x_{1}, x_{2}, \ldots, x_{k} $ such that $x_{j} \in \mathbb {N}_{i_j}$, where $j = 1,2, \ldots,k$.

Let $J_{k, \, n} (N)$ be the number of solutions of (\ref{Var1}) in an arbitrary numbers $ x_ {1},  x_{2}, \ldots, x_ {k} $ and $ I_ {k, \, n} (N) $ -- the number of solutions (\ref{Var1}) in the numbers $ x_ {1},x_ {2}, \ldots,$  $x_ {k}$ from $\mathbb{N}_{0}$. Our main result is the following:

\begin{theorem}\label{t1}
Let $n\geqslant 3$,
$$
k_{0} = \left\{
          \begin{array}{ll}
            2^{n}, & \hbox{if $3 \leqslant n \leqslant 10$;} \\
            2 \, [n ^ {2} (\log n + \log\log n +4)], & \hbox{if {~~~~~}$n> 10$.}
          \end{array}
        \right.
$$

Let $k\geqslant k_{0}$. Then the formula
$$
I_ {k, \, n} (N) = 2^{-k} J_ {k, \, n} (N) + O (N^{k/n-1- \Delta / n}),
$$
where $ \Delta = c_ {1} (n^{2}\log n)^{-1}$, $c_ {1}> 0$, holds.
\end{theorem}

Recall that  $J_{k, \, n}(N)\asymp N^{k/n-1}$ (see eg \cite{AAK}, chapter ~ XI).

To prove the main theorem we need several lemmas.

\begin{center}
{\bf {2. Lemmas}}
\end{center}
\bigskip

\begin{lemma}\label{HLK}(Hua Loo-Keng)

Let
$$
S(\alpha) =\sum\limits_{x \leqslant P} e ^ {2 \pi i \alpha x ^ {n}}.
$$

Then for $1\leqslant j \leqslant n $
$$
\int_{0}^{1} | S (\alpha) |^{2^{j}} \, d \alpha \ll P^{2^{j}-j + \varepsilon},
$$
where $\varepsilon> 0 $ ia an arbitrarily small number.
\end{lemma}

 For the proof see \cite{Von}, p. 20.

\begin{lemma}\label{Vin}(Vinogradov)

Let $n\geqslant 4$, $k\geqslant 2k_ {0}$, where $k_{0}=[n ^ {2} (\log n + \log \log n +4)]$.
Then we have
$$
\int_{0}^{1}| S (\alpha) |^{k} \, d \alpha \ll P^{k-n}.
$$
\end {lemma}

For the proof see \cite{Vinogradov}, p. 84.

\begin{lemma}\label{I1}
Let $ h $ be a natural number, $h \leqslant H$, $ H\leqslant X/4$ . Define the sum $S (X, h)$
$$
S (X, h) = \sum\limits_{n \leqslant X}\varepsilon (n)\varepsilon(n + h).
$$
Then the estimate
$$
\sum\limits_{h = 1}^H|S(X, h)|=O(X \, H^{\mu}),
$$
where $ \mu = \frac {\ln 3,5} {\ln 4} = 0.903 ... $, holds. The constant implied in sign $O$ is  absolute.
\end{lemma}

\begin{proof}

Let

$$ V (X, h) = \sum\limits_{n \leqslant X}\varepsilon (n) \left (\varepsilon (n + h) + \varepsilon (n + h +1) \right).
$$

Consider the binary expansion of $h$:

$$
h = \omega_{0} +2\omega_{1} + \ldots +2^{k-1} \omega_{k-1} +2^{k},
$$
where $\omega_{j}= 0 ,\, 1 \, \, \, (j = 0,1, \ldots, k-1)$.

For $0\leqslant j <k$ define the numbers $h_{j}$ and $s_{j}$:
$$
h_{j} =\omega_{j} +2 \omega_{j +1} + \ldots +2^{k-1-j} \omega_ {k-1} +2^{k-j},
$$
$$
s_{j} = 1-2 \omega_ {j}.
$$

Dividing each of the sums
$$
S(X2^{-j}, h_{j}), \, \, V(X2^{-j}, h_{j})
$$
into two sums -- for even and odd $n$ -- we get:

\begin{equation}\label {l1}
 S (X2^{-j}, h_{j}) = (1 + s_ {j}) S (X2^{-j-1}, h_{j +1}) + \frac{s_ {j} - 1}{2} V (X2 {-j-1}, h_{j +1}) + \theta_{j}
\end{equation}

\begin{equation}\label {l2}
 V(X2^{-j}, h_{j}) = 2s_{j} S (X2^{-j-1}, h_{j +1})-s_{j} V(X2^{-j- 1}, h_{j +1}) + \theta_
{j }^{'},
\end {equation}
where $|\theta_ {j} |$, $| \theta_{j }^{'}| \leqslant 1$.

Let $\alpha_{0}= 1$, $\beta_{0}=1$. Then we have
$$
S (X, h) = \alpha_{0} S (X, h_{0}) + \beta_{0} V(X, h_{0}).
$$

According to  (\ref{l1})  and  (\ref{l2})  for any $ j = 1,2, \ldots, k-1 $ we have
$$
 S(X, h) = \alpha_{j} S (X2^{-j}, h_{j}) + \beta_{j} V(X2^{-j}, h_{j}) + O (|\alpha_{j}| + | \beta_{j}|),
$$
where
  \begin{equation}\label{l3}
 \alpha_{j +1} = (1 + s_{j})\alpha_{j} +2 s_{j} \beta_{j},
   \end {equation}
\begin{equation}\label{l4}
    \beta_{j +1} = \frac{s_{j} -1}{2}\alpha_{j}-s_{j} \beta_{j}.
\end{equation}

It follows from $(\ref {l3})$ that
\begin{equation}\label {l5}
   \beta_{j} =\frac{\alpha_{j +1} - (1 + s_ {j}) \alpha_ {j}}{2s_{j}}.
\end {equation}

Substitute  (\ref{l5}) in  (\ref{l4}):
 \begin{equation}\label {l6}
    \beta_{j +1} =- \frac{\alpha_{j +1}}{2} + s_{j} \alpha_{j}.
 \end{equation}

Substitute $j +1$ instead of $j$ in  (\ref{l5}):
$$
\beta_{j +1} =\frac{\alpha_{j +2} - (1 + s_{j +1}) \alpha_{j +1}} {2s_{j +1}}.
$$

From this and from  (\ref{l6})  we obtain:

\begin{equation}\label{l7}
    \alpha_{j +2} - \alpha_{j +1} = 2s_{j} s_{j +1}\alpha_{j}.
\end{equation}

Furthermore,
$$
\alpha_{j +3} - \alpha_{j +2} = 2s_{j +1} s_{j +2}\alpha_{j +1},
$$
\begin{equation}\label{l8}
   \alpha_{j +3} = (1 +2 s_{j +1} s_{j +2}) \alpha_{j +1} +2 s_{j} s_{j +1} \alpha_{j}.
\end {equation}

We estimate $|\alpha_{j}|$ from above equality.

First, since $\alpha_{0} = 1$, $\alpha_{1} = 1 + s_{0}$,    we obtain from (\ref{l7}) by induction that
\begin{equation}\label{l9}
    |\alpha_{j}|\leqslant 2^{j}, {~~~~} j = 0,1, \ldots, k-1.
\end{equation}

 This estimate, in some cases can be improved.

Let $j_{1}$ be the smallest odd number in the interval $[1, k-3]$ such that $s_{j +1} s_{j +2} =-1$. Then from (\ref{l8})  and  (\ref{l9}) we have:
\begin{equation}\label{l10}
    |\alpha_{j_ {1} +3}| \leqslant \frac{1}{2} \, 2^{j_ {1} +3};
\end{equation}

Further, from (\ref{l7}), (\ref{l9}) and  (\ref{l10})  we obtain:
$$
|\alpha_{j_{1}+4}|\leqslant|\alpha_{j_{1}+3}|+2|\alpha_{j_{1}+2}|\leqslant\frac{1}{2}\,2^{j_{1}+3}+2^{j_{1}+3}=\frac{3}{4}\,2^{j_{1}+4};
$$

Now, since $|\alpha_{j_ {1} +3}|\leqslant\frac{3}{4}\, 2^{j_ {1} +3}$, $|\alpha_{j_{1} + 4}|\leqslant \frac{3}{4} \, 2^{j_{1}+4}$,   we get that from (\ref{l7})
\begin{equation}\label{l11}
    |\alpha_{j_{1} + t}|\leqslant \frac{3}{4} \, 2^{j_{1} + t} {~~~~} \text {when }{~~~~ 4} \leqslant t.
\end{equation}

Let $j_2$ be the smallest odd number in the interval $[j_{1} +2, k-3]$ such that $s_{j_{2} +1} s_{j_{2} +2} =- 1$.

Arguing as in the derivation of (\ref{l10})  and  (\ref{l11}), we obtain:
$$
2^{j_{2} + t }{~~~} \text {when }{~~~} 4 \leqslant t.
$$

Continuing this process, we get $|\alpha_{k}|\leqslant \left(\frac{3}{4}\right)^{\varkappa (h)} 2^{k}$, where $\varkappa(h)$ is the number of sign changes in the pairs of numbers
$(s_{1}, s_{2})$,  $(s_{3}, s_{4})$, \ldots, $(s_{2j_{0}-1}, s_{2j_{0}})$
where $2j_{0} -1$ is the greatest integer not exceeding $ k-3$.

Let $\varkappa_{1,2} (h_{1})$ be the number 1 and 2 in the decomposition of $h_{1}$ for the base 4.

Then $\varkappa(h) \geqslant \varkappa_ {1,2} (h_ {1}) -1 $.

In fact,
$$
h_{1} =\sum\limits_{1 \leqslant 2j-1 \leqslant k-3} \eta_{j} 4^{j-1},
$$
where $\eta_{j} =\omega_{2j-1} +2 \omega_{2j}$; 1 is subtracted from the $\varkappa_{1,2} (h_ {1})$, so that 1 or 2 most significant digit can not be ignored because of the condition $2j-1 \leqslant k-3$.

Thus,
$$
|\alpha_{k} |\leqslant \left(\frac{3}{4}\right)^{\varkappa_{1,2} (h_{1})}\cdot\frac{4}{3}\cdot 2^{k}.
$$

From this and from (\ref{l5}) implies that
$$
|\beta_{k}|\leqslant \left(\frac{3}{4} \right)^{\varkappa_{1,2} (h_{1})} 2^{k +1},
$$
so
$$
|S(X, h)|\ll X \left(\frac{3}{4}\right)^{\varkappa_{1,2} (h_{1})}.
$$

Let $t$ be the smallest positive integer such that $H<4^{t}$.

Estimate
$$
\sum\limits_{h_{1} = 0}^{4^{t} -1} \left(\frac{3}{4}\right)^{\varkappa_{1,2} (h_{1})}.
$$

We have
$$
\sum\limits_{h_{1} = 0}^{4^{t} -1} \left(\frac{3}{4}\right)^{\varkappa_{1,2} (h_{1})} = \sum\limits_{s = 0}^{t}\left(\frac{3}{4}\right)^{s}\sum\limits_{h_{1} = 0, \varkappa_{a , 2} (h_{1}) = s}^{4^{t} -1}  1=
$$
$$
=\sum\limits_{h_{1} = 0}^{4^{t} -1} \left(\frac{3}{4}\right)^{\varkappa_{1,2} (h_{1})} =
\sum\limits_{s = 0}^{t}\left(\frac{3}{4}\right)^{s}\sum_{\substack{h_ {1} = 0 \ \ \varkappa_ {1, 2} (h_{1}) = s}}^{4^{t} -1}  1=
$$
$$
=\sum\limits_{s=0}^{t}\left(\frac{3}{4}\right)^{s} \sum\limits_{l=0}^{s}\binom{t}{l}\binom{t-l}{s-l}2^{t-s}= \sum\limits_{s=0}^{t}\left(\frac{3}{4}\right)^{s} 2^{t-s}\binom{t}{s} \sum\limits_{l=0}^{s}\binom{s}{l}=
$$
$$
=2^{t}\sum\limits_{s=0}^{t}\binom{t}{s}\left(\frac{3}{4}\right)^{s} =\left(\frac{7}{2}\right)^{t}\leqslant H^{\frac{\ln 3,5}{\ln 4}}=H^{\mu}.
$$
\end{proof}

\begin{lemma}\label{lemma4}
Let $\eta>0$. Let $g (x)$ be a polynomial of degree $n\geqslant 10$ with leading coefficient $\alpha$. Let $\alpha = \frac{a}{q} + \frac{\theta}{q ^ {2}}$, $(a, \, q) = 1$, $|\theta|\leqslant 1$ , $P^{\eta} \leqslant q \leqslant P^{n-\eta}$. Then
$$ \left|\sum\limits_{x\leqslant P}e^{2\pi i g(x)}\right|\ll_{\eta}P^{1-c_{1}(n^{2}\log n)^{-1}},
$$
where $ c_{1} = c_{1} (\eta, \ n)> 0$.
\end{lemma}

\begin{proof}
Lemma 4 is different from Theorem 2 of \cite{AAK}, p. 190 that the condition $P^{1 / 4}\leqslant q \leqslant P^{n-1 / 4}$ replaced by $P^{\eta}\leqslant q \leqslant P^{n-\eta}$. Proof of Lemma 4 and the theorem is essentially the same, only the parameter $\tau$ must be chosen so: $\tau = [C (\eta) n \log n]$, where $C(\eta)> 0$ -- just a large number.
\end{proof}

\begin{lemma}\label{Veyl} (Weyl Inequality)
Let $g(x)$ be a polynomial of degree $n$, $2\leqslant n \leqslant  10$, with leading coefficient $ \alpha$, $\alpha = \frac{a}{q} + \frac{a}{q^{2}} $, $ (a, \, q) = 1 $, $ |\theta| \leqslant 1 $.

Then
$$
\big|\sum\limits_{x \leqslant P} e^{2 \pi ig (x)} \big| \ll P^{1 + \varepsilon} \big(q^{-1} + P^{ -1} + qP^{-n} \big)^{2^{-n +1}}.
$$
\end {lemma}

For the proof see \cite{Von}, p. 19.

\begin{center}
{\bf {3. Proof of Theorem}}
\end {center}
\bigskip

\textbf{3.1} We first prove that the sum
$$
W(\alpha)=\sum\limits_{x \leqslant P}\varepsilon(x) e^{2\pi i \alpha x^{n}}
$$
for any $\alpha\in\mathbb{R}$ satisfies the estimate
$$
|W(\alpha)|\ll_{n}P^{1-c_{1} (n^{2}\log n)^{-1}},
$$
where $c_{1}> 0$ is a constant.

Approximate $\alpha$ with a rational number:
$\alpha =\frac{a}{q}+\frac{\theta}{q \tau}$, $(a, \, q) = 1$, $|\theta| \leqslant 1$, $1 \leqslant q \leqslant\tau=P^{n-1-1/2000}$.

Cases where $q \leqslant P^{0,001}$ and $P^{0,001}<q \leqslant \tau $, are considered in different ways.

Assume first that $P^{0,001} <q \leqslant \tau$. Define integers $H$ and $K$:
$$
H = [P^{1/4000n}],{~~~~~~} 2^{K-1} <P^{1/2000n} \leqslant 2^{K}.
$$
We apply the well-known inequality, whose proof is given, for example, \cite{VoroninKaratsuba}, p. 361:
$$
|W(\alpha)|^{2}\ll\frac{P^{2}}{H}+\frac{P}{H}\sum\limits_{1\leqslant h<H}\big|\sum\limits_{x\leqslant P-h}\varepsilon(x)\varepsilon(x+h)e^{2\pi i \alpha ((x+h)^{n}-x^{n})}\big|\ll
$$
$$
\ll\frac{P^{2}}{H}+PH+\frac{P}{H}\sum\limits_{1\leqslant h<H}\big|\sum\limits_{x\leqslant P}\varepsilon(x)\varepsilon(x+h)e^{2\pi i \alpha ((x+h)^{n}-x^{n})}\big|.
$$

Put in the last sum ${a}/{q}$ instead of $ \alpha $ and estimate the error occurring at the same time.

We have
$$
e^{2\pi i \alpha ((x+h)^{n}-x^{n})}=e^{2\pi i \frac{a}{q}((x+h)^{n}-x^{n})}e^{2\pi i z((x+h)^{n}-x^{n})}=
$$
$$
=e^{2\pi i \frac{a}{q}((x+h)^{n}-x^{n})}\big(1+O(|z|h P^{n-1})\big)=
$$
$$
=e^{2\pi i \frac{a}{q}((x+h)^{n}-x^{n})}\big(1+O(P^{-1/4000})\big),
$$
since
$$
|z|h<\frac{H}{q\tau}\leqslant\frac{1}{q}P^{-n+1+1/2000+1/4000},{~~~}q>P^{1/1000}.
$$

Hence we have
$$
|W(\alpha)|^{2}\ll\frac{P^{2}}{H}+PH+P^{2-1/4000}+$$
$$
+\frac{P}{H}\sum\limits_{1\leqslant h<H}\big|\sum\limits_{x\leqslant P}\varepsilon(x)\varepsilon(x+h)e^{2\pi i \frac{a}{q} ((x+h)^{n}-x^{n})}\big|.
$$

We split the last sum in arithmetic progressions with the difference $2^{K}$:
$$
\sum\limits_{x\leqslant P}\varepsilon(x)\varepsilon(x+h)e^{2\pi i \frac{a}{q} ((x+h)^{n}-x^{n})}=
$$
$$
=\sum\limits_{m=0}^{2^{K}-1}\sum\limits_{1\leqslant 2^{K}y+m\leqslant P}\varepsilon(2^{K}y+m)\varepsilon(2^{K}y+m+h)e^{2\pi i \frac{a}{q} ((2^{K}y+m+h)^{n}-(2^{K}y+m)^{n})}=
$$
$$
=\sum\limits_{m=0}^{2^{K}-h-1}\sum\limits_{y\leqslant P2^{-K}}\varepsilon(2^{K}y+m)\varepsilon(2^{K}y+m+h)e^{2\pi i \frac{a}{q} ((2^{K}y+m+h)^{n}-(2^{K}y+m)^{n})}+$$
$$
+O(2^{K})+O(PH2^{-K}).
$$

By definition of $\varepsilon(n)$ for $0\leqslant m <2^{K}+ h $, we have
$$
\varepsilon(2^{K} y+ m)\varepsilon(2^{K} y + m + h) = \varepsilon(m)\varepsilon(m + h),
$$
so
$$
\big|\sum\limits_{x\leqslant P}\varepsilon(x)\varepsilon(x+h)e^{2\pi i \frac{a}{q} ((x+h)^{n}-x^{n})}\big|\ll
$$
$$
\ll\sum\limits_{m=0}^{2^{K}-1}\big|\sum\limits_{y\leqslant P2^{-K}}e^{2\pi i \frac{a}{q} ((2^{K}y+m+h)^{n}-(2^{K}y+m)^{n})}\big|+2^{K}+PH2^{-K}.
$$

For fixed $m$ and $h$ function
$$
\frac{a}{q}\big((2^{K}y+m+h)^{n}-(2^{K}y+m)^{n}\big)
$$
is a polynomial of $y$ of degree $n-1$ with leading coefficient $\frac{a}{q}nh2^{K(n-1)}$.
If this rational fraction is cancellable, then perform the reduction and
$$
\frac{a}{q}nh2^{K(n-1)}=\frac{a_{1}}{q_{1}},{~~~}\text{где}{~~~}(a_{1},\,q_{1})=1.
$$

Since $(a,\,q)=1$, $q_{1}\gg q2^{-Kn}>P^{1/2000}$, because $q>P^{0,001}$, $2^{Kn}\geqslant P^{1/2000}$.

Thus,
\begin{equation}\label{teo2}
    P^{1/2000}<q_{1}\leqslant\tau=P^{n-1-1/2000}.
\end{equation}

Thus, it suffices to estimate Weyl's sum  of a special type
$$
\sum\limits_{y\leqslant P2^{-K}}e^{2\pi i g(y)},
$$
where $g(y)$ is a polynomial of degree $n-1$ and leading coefficient $a_{1}/q_{1} $, where $(a_{1}, \, q_{1}) = 1$, and $ q_{1}$ satisfies $(\ref {teo2})$.

Estimating the sum at $n\leqslant 11 $ by Lemma 5, and when $n> 11$ -- by Lemma 4, we obtain:
$$
\big|\sum\limits_{x\leqslant P}\varepsilon(x)\varepsilon(x+h)e^{2\pi i \alpha ((x+h)^{n}-x^{n})}\big|\ll P^{1-2c_{1}(n^{2}\log n)^{-1}},{~~~}(c_{1}>0),
$$
and so
$$
|W(\alpha)|\ll P^{1-c_{1}(n^{2}\log n)^{-1}}
$$
in the case of $P^{0,001} <q \leqslant P^{n-1-1/2000}$.

Now let $1\leqslant q \leqslant P^{0,001}$. Transform $ W(\alpha)$:
$$
W(\alpha)=\sum\limits_{x\leqslant P}\varepsilon(x)e^{2 \pi i \frac{a}{q}x^{n}}e^{2\pi i z x^{n}}= \sum\limits_{l=0}^{q-1}e^{2 \pi i \frac{al^{n}}{q}} \sum_{\substack {x\leqslant P, \\ x\equiv l(\bmod q)}}\varepsilon(x)e^{2 \pi i z x^{n}}=
$$
$$
=\sum\limits_{l=0}^{q-1}e^{2 \pi i \frac{al^{n}}{q}} \sum\limits_{x\leqslant P}\varepsilon(x)e^{2 \pi i z x^{n}}\frac{1}{q}\sum\limits_{b=0}^{q-1}e^{2 \pi i \frac{b(l-x)}{q}}=
$$
$$
=\frac{1}{q}\sum\limits_{b=0}^{q-1}\Big(\sum\limits_{l=0}^{q-1}e^{2 \pi i \frac{al^{n}+bl}{q}}\Big)W(z,\,b),
$$
where
$$
W(z,\,b)=\sum\limits_{x\leqslant P}\varepsilon(x)e^{2 \pi i (zx^{n}-\frac{bx}{q})}.
$$

We use the Cauchy inequality:
$$
|W(z,\,b)|^{2}\leqslant \frac{1}{q}\sum\limits_{b=0}^{q-1}\big|\sum\limits_{l=0}^{q-1}e^{2 \pi i \frac{al^{n}+bl}{q}}\big|^{2}\,\,\big|W(z,\,b)\big|^{2}.
$$

Let $H_{1}=[P^{1/30}]$. Then
$$
|W(z,\,b)|^{2}\ll \frac{P^{2}}{H_{1}}+PH_{1}+\frac{P}{H_{1}}\sum\limits_{1\leqslant h< H_{1}}\big|\sum\limits_{x\leqslant P}\varepsilon(x)\varepsilon(x+h)e^{2 \pi i z ((x+h)^{n}-x^{n})}\big|.
$$

We define a natural number $ K_{1} $ from the inequalities
$$
2^{K_{1}-1}<P^{1/15}\leqslant 2^{K_{1}}.
$$

We divide  $ x $ in arithmetic progression with difference of $ 2^{K_{1}} $ and obtain:
$$
\big|\sum\limits_{x\leqslant P}\varepsilon(x)\varepsilon(x+h)e^{2\pi i z ((x+h)^{n}-x^{n})}\big|\ll
$$
$$
\ll \big|\sum\limits_{m=0}^{2^{K_{1}}-h-1}\sum\limits_{y\leqslant P2^{-K_{1}}}\varepsilon(2^{K_{1}}y+m)\varepsilon(2^{K_{1}}y+m+h)
e^{2\pi i z ((2^{K_{1}}y+m+h)^{n}-(2^{K_{1}}y+m)^{n})}\big|+
$$
$$
+2^{K_{1}}+PH_{1}2^{-K_{1}}.
$$

Again we use the identity
$$
\varepsilon(2^{K_{1}}y+m)\varepsilon(2^{K_{1}}y+m+h)=\varepsilon(m)\varepsilon(m+h),
$$
valid for $0\leqslant m<2^{K_{1}}-h$. We obtain
$$
\big|\sum\limits_{x\leqslant P}\varepsilon(x)\varepsilon(x+h)e^{2\pi i z ((x+h)^{n}-x^{n})}\big|\ll
$$
$$
\ll \big|\sum\limits_{m=0}^{2^{K_{1}}-1}\varepsilon(m)\varepsilon(m+h)\sum\limits_{y\leqslant P2^{-K_{1}}}e^{2\pi i z ((2^{K_{1}}y+m+h)^{n}-(2^{K_{1}}y+m)^{n})}\big|+$$
$$
+2^{K_{1}}+PH_{1}2^{-K_{1}}.
$$

Our next goal is to show that the sum over $y$ "depends only weak" on $ m $ and $ h $.

The following equalities hold:
$$
(2^{K_{1}}y+m+h)^{n}-(2^{K_{1}}y+m)^{n} =
$$
$$=h\big((2^{K_{1}}y+m+h)^{n-1}+(2^{K_{1}}y+m+h)^{n-2}(2^{K_{1}}y+m)+ \ldots +(2^{K_{1}}y+m)^{n-1}\big);
$$
$$
(2^{K_{1}}y+m+h)^{j}=(2^{K_{1}}y)^{j}+O(2^{K_{1}}P^{j-1});{~~~}(1\leqslant j\leqslant n-1)
$$
$$
(2^{K_{1}}y+m)^{j}=(2^{K_{1}}y)^{j}+O(2^{K_{1}}P^{j-1}).{~~~}(1\leqslant j\leqslant n-1)
$$

From these equations it follows that
$$
(2^{K_{1}}y+m+h)^{n}-(2^{K_{1}}y+m)^{n} =hn(y2^{K_{1}})^{n-1}+O(H_{1}2^{2K_{1}}P^{n-2}),
$$
$$
e^{2 \pi i z ((2^{K_{1}}y+m+h)^{n}-(2^{K_{1}}y+m)^{n}})=e^{2 \pi i z hn (y2^{K_{1}})^{n-1}}1+O\big(|z|H_{1}2^{2K_{1}}P^{n-2}\big).
$$

By definition, $|z|\leqslant \frac{1}{\tau}=P^{n-1+1/2000}$, $H_{1}\leqslant P^{1/30}$, $2^{K_{1}}\ll P^{1/15}$, so
$$
|z|H_{1}2^{K_{1}}P^{n-2}\ll P^{-1+1/2000+2/15+1/30}\ll P^{-1/2}.
$$

We have obtained the inequality
$$
\Big|\sum\limits_{x\leqslant P}\varepsilon(x)\varepsilon(x+h)e^{2\pi i z ((x+h)^{n}-x^{n})}\Big|\ll P2^{-K_{1}}\Big|\sum\limits_{m=0}^{2^{K_{1}}-1}\varepsilon(m)\varepsilon(m+h)\Big|+
$$
$$
+P^{1/2}+2^{K_{1}}+PH2^{-K_{1}}.
$$

Returning to the estimate $ W (\alpha, \, b) $, we obtain:
$$
|W(\alpha,\,b)|^{2}\ll\frac{P^{2}}{H_{1}}+PH_{1}+\frac{P^{2}}{H_{1}2^{K_{1}}} \sum\limits_{1\leqslant h\leqslant H_{1}}\big|\sum_{m\leqslant 2^{K_{1}}}\varepsilon(m)\varepsilon(m+h)\big|+$$
$$+\big(2^{K_{1}}+PH_{1}2^{-K_{1}}\big)P.
$$

Applying Lemma 3, we arrive at
$$
|W(\alpha,\,b)|^{2}\ll\frac{P^{2}}{H_{1}}+PH_{1}+P2^{K_{1}}+ P^{2}H_{1}2^{-K_{1}}+\frac{P^{2}}{H_{1}^{1-\mu}}\ll \frac{P^{2}}{H_{1}^{1-\mu}},
$$
where $\mu=0,901\ldots$

Finally, since
$$
\frac{1}{q}\sum\limits_{b=0}^{q-1}\Big|\sum_{l=0}^{q-1}e^{2 \pi i \frac{al^{n}+bl}{q}}\Big|^{2}=\sum_{l_{1},\,l_{2}=0}^{q-1}e^{2 \pi i \frac{a(l_{1}^{n}-l_{2}^{n})}{q}}\frac{1}{q}\sum\limits_{b=0}^{q-1} e^{2 \pi i\frac{b(l_{1}-l_{2})}{q}}=q,
$$
we obtain:
$$
|W(\alpha)|^{2}\ll\frac{P^{2}q}{H_{1}^{1-\mu}}\ll    P^{2}P^{0,002-(1-\mu)/15}\ll P^{2-0,05}.
$$

Thus, for any $\alpha\in \mathbb{R}$
$$
|W(\alpha)|\ll P^{1-\frac{c_{1}}{n^{2}\log n}},
$$
where $c_{1}$ is a constant.

\textbf{3.2} Conclusion of the proof of the theorem.

Let  $P=N^{1/n}$.

Since
$$
\frac{1+\varepsilon(x)}{2}=\left\{
  \begin{array}{ll}
    1, & \text{если} {~~~} x\in \mathbb{N}_{0}; \\
    0, & \text{если} {~~~} x\in \mathbb{N}_{1},
  \end{array}
\right.
$$
we have
$$
I_{k,\,n}(N)=\int_{0}^{1}\Big(\frac{S(\alpha)+W(\alpha)}{2}\Big)^{k}e^{-2 \pi i \alpha N}\,d \alpha,
$$
where
$$
S(\alpha)=\sum\limits_{x\leqslant P}e^{2 \pi i \alpha x^{n}}.
$$

Hence we have
$$
I_{k,\,n}(N)=2^{-k}J_{k,\,n}(N)+O(R),
$$
where

$$
R=\sum_{l=1}^{k}\binom{k}{l}R_{l},{~~~}R_{l}=\int_{0}^{1}|W(\alpha)|^{l}|S(\alpha)|^{k-l}\,d\alpha.
$$

Note that for positive $ k_{1} $ and $ l_{1} $
\begin{equation}\label{teo3}
    \int_{0}^{1}|W(\alpha)|^{2l_{1}}|S(\alpha)|^{2k_{1}}\,d\alpha\leqslant \int_{0}^{1}|S(\alpha)|^{2(k_{1}+l_{1})}\,d\alpha.
\end{equation}

Indeed,
$$
\int_{0}^{1}|W(\alpha)|^{2l_{1}}|S(\alpha)|^{2k_{1}}\,d\alpha=
$$
$$=\sum_{x_{1}\leqslant P}\varepsilon(x_{1})\ldots\sum_{x_{l_{1}}\leqslant P}\varepsilon(x_{l_{1}}) \sum_{y_{1}\leqslant P}\varepsilon(y_{1})\ldots\sum_{y_{l_{1}}\leqslant P}\varepsilon(y_{l_{1}})\times
$$
$$
\times\int_{0}^{1}|S(\alpha)|^{2k_{1}}e^{2 \pi i \alpha (x_{1}^{n}+\ldots+x_{l_{1}}^{n}-y_{1}^{n}-\ldots-y_{l_{1}}^{n})}\,d\alpha\leqslant
$$
$$
\leqslant \sum_{x_{1}\leqslant P}\ldots\sum_{x_{l_{1}}\leqslant P}\sum_{y_{1}\leqslant P}\ldots\sum_{y_{l_{1}}\leqslant P}\Big|\int_{0}^{1}|S(\alpha)|^{2k_{1}}e^{2\pi i \alpha (x_{1}^{n}+\ldots+x_{l_{1}}^{n}-y_{1}^{n}-\ldots-y_{l_{1}}^{n})}\,d\alpha\Big|.
$$

But
$$
\int_{0}^{1}|S(\alpha)|^{2k_{1}}e^{2\pi i \alpha (x_{1}^{n}+\ldots+x_{l_{1}}^{n}-y_{1}^{n}-\ldots-y_{l_{1}}^{n})}\,d\alpha\geqslant 0,
$$
so
$$
\leqslant \sum_{x_{1}\leqslant P}\ldots\sum_{x_{l_{1}}\leqslant P}\sum_{y_{1}\leqslant P}\ldots\sum_{y_{l_{1}}\leqslant P}\Big|\int_{0}^{1}|S(\alpha)|^{2k_{1}}e^{2\pi i \alpha (x_{1}^{n}+\ldots+x_{l_{1}}^{n}-y_{1}^{n}-\ldots-y_{l_{1}}^{n})}\,d\alpha\Big|=
$$
$$
=\int_{0}^{1}|S(\alpha)|^{2(k_{1}+l_{1})}\,d\alpha.
$$

Using the inequality (\ref{teo3}), we estimate $ R_{l} $ for $ 1 \leqslant l \leqslant k $. We consider separately the four possible cases.

$ 1^{\circ} $ $ l $ is odd, $ (k-l) $ -- an even number. Then
$$
R_{l}\leqslant |W(\alpha_{0})|\int_{0}^{1}|S(\alpha)|^{k-1}\,d\alpha,
$$
where
$$
|W(\alpha_{0})|=\max\limits_{0\leqslant \alpha\leqslant 1}|W(\alpha)|.
$$

$ 2^{\circ} $ $ l $ is odd, $ (k-l) $ is odd. Then
$$
R_{l}\leqslant |W(\alpha_{0})|P\int_{0}^{1}|S(\alpha)|^{k-2}\,d\alpha.
$$

$ 3^{\circ} $ $ l $ is an even number, $ (k-l) $ is an even number. Then
$$
R_{l}\leqslant |W(\alpha_{0})|^{2}\int_{0}^{1}|S(\alpha)|^{k-2}\,d\alpha.
$$

$ 4^{\circ} $ $ l $ is an even number, $ (k-l) $ is odd. Then
$$
R_{l}\leqslant |W(\alpha_{0})|^{2}P\int_{0}^{1}|S(\alpha)|^{k-3}\,d\alpha.
$$

Hence, for any $l\in [1,\,k]$, we have
$$
R_{l}\leqslant |W(\alpha_{0})|P^{2}\int_{0}^{1}|S(\alpha)|^{k-3}\,d\alpha.
$$

By hypothesis, $k\geqslant k_{0}+3$. Applying for $3\leqslant n \leqslant10$ Lemma ~ 1, and $n>10$ -- Lemma 2, we obtain:
$$
R_{l}\ll |W(\alpha_{0})|P^{k-1-n+\varepsilon},
$$
where $\varepsilon >0$.

Since
$$
|W(\alpha_{0})|\ll P^{1-c_{1}(n^{2}\log n)^{-1}},
$$
choosing
$$
\varepsilon =\frac{c_{1}}{2n^{2}\log n},
$$
we get
$$
R\ll P^{k-n-\Delta},
$$
where
$$
\Delta=\frac{c_{1}}{2n^{2}\log n}.
$$
The theorem is proved.

\pagebreak

\end{document}